\documentclass[12pt, a4paper]{amsart}
\usepackage[latin1]{inputenc}
\usepackage{mathtools}
\usepackage{amsmath}
\usepackage{amsthm}
\usepackage{amsfonts}
\usepackage{amssymb}
\usepackage{a4wide}

\DeclareMathOperator{\SL}{SL}
\DeclareMathOperator{\id}{id}
\DeclareMathOperator{\N}{\mathbb{N}}
\DeclareMathOperator{\Z}{\mathbb{Z}}
\DeclareMathOperator{\R}{\mathbb{R}}
\DeclareMathOperator{\C}{\mathbb{C}}

\renewcommand{\H}{\mathbb{H}}

\newtheorem{Theorem}{Theorem}[section]
\newtheorem{Lemma}[Theorem]{Lemma}
 
\newtheorem{Corollary}[Theorem]{Corollary}
	
\theoremstyle{definition}

\author{Claudia Alfes-Neumann}
\address{Mathematical Institute, Paderborn University, Warburger Str. 100,
D-33098 Paderborn, Germany}
\email{alfes@math.uni-paderborn.de}

\address{Mathematical Institute, University of Cologne, Weyertal 86-90, D--50931 Cologne, Germany}

\author{Markus Schwagenscheidt}
\email{mschwage@math.uni-koeln.de}
\title{Identities of Cycle Integrals of Weak Maass Forms}

\thanks{The research of the third author is supported by the SFB-TRR 191 \lq Symplectic Structures in Geometry, Algebra and Dynamics\rq, funded by the DFG}

\begin{document} 

	\begin{abstract}
		We prove identities between cycle integrals of non-holomorphic modular forms arising from applications of various differential operators to weak Maass forms.
	\end{abstract}
	
	\maketitle

	\section{Introduction}
	
	While investigating the Shintani lift of weakly holomorphic modular forms, Bringmann, Guerzhoy and Kane \cite{bgk1,bgk2} found surprising identities between cycle integrals of weakly holomorphic modular forms and cusp forms of the same weight. More precisely, they showed that if $f$ is a weakly holomorphic modular form of weight $2k \in 2\N$ for $\Gamma=\SL_{2}(\Z)$ which is orthogonal to cusp forms and whose constant term in the Fourier expansion vanishes, then there exists a cusp form $g$ of weight $2k$ such that the identity
	\begin{align}\label{eq bgk formula 1}
	\int_{C_{Q}}f(z)Q(z,1)^{k-1}dz =-\frac{(2k-2)!}{(4\pi)^{2k-1}}\overline{\int_{C_{Q}}g(z)Q(z,1)^{k-1}dz}
	\end{align}
	of cycle integrals holds for any integral binary quadratic form $Q$ of positive non-square discriminant, where $C_{Q}$ is the geodesic associated to $Q$, see Section~\ref{section cycle integrals}. Although the above identity only involves (weakly) holomorphic modular forms, the proof crucially uses the theory of harmonic weak Maass forms, which was first developed systematically by Bruinier and Funke \cite{bf}. Namely, by results of Bruinier, Ono and Rhoades \cite{bor}, for every weakly holomorphic modular form $f$ of weight $2k$ as above there exists a harmonic weak Maass form $F$ of weight $2-2k$ such that $f = \mathcal{D}^{2k-1}F$ and such that $g=\xi_{2-2k}F$ is a cusp form of weight $2k$, with the differential operators $\xi_{2-2k} = 2i y^{2-2k}\overline{\frac{\partial}{\partial \overline{z}}}$ and $\mathcal{D} = \frac{1}{2\pi i}\frac{\partial}{\partial z}$, $z = x+iy \in \H$. The authors of \cite{bgk2} actually showed that the identity
	\begin{align}\label{eq bgk formula}
	\int_{C_{Q}}(\mathcal{D}^{2k-1}F)(z)Q(z,1)^{k-1}dz = -\frac{(2k-2)!}{(4\pi)^{2k-1}}\overline{\int_{C_{Q}}(\xi_{2-2k}F)(z)Q(z,1)^{k-1}dz}
	\end{align}
	holds for all binary quadratic forms $Q$ as above. This implies the formula \eqref{eq bgk formula 1}. For the proof of \eqref{eq bgk formula}, the authors of \cite{bgk2} defined (regularized) periods of weakly holomorphic modular forms and generalized identities between cycle integrals and periods of cusp forms proved by Kohnen and Zagier \cite{kohnenzagier} to weakly holomorphic modular forms. Jens Funke informed us that he has obtained another proof of the identity \eqref{eq bgk formula} by comparing the cohomology classes of $\mathcal{D}^{2k-1}F$ and $\xi_{2-2k}F$. 
	
	In the present note, we generalize the above identities by replacing $f$ and $g$ by non-holomorphic modular forms arising from applications of various differential operators to weak Maass forms. Furthermore, we prove our identities by direct calculations using explicit parametrizations of cycle integrals and commutation relations of differential operators, and thereby obtain refinements and shorter proofs of the above identities of Bringmann, Guerzhoy and Kane.
	
	To simplify the notation, we define the cycle integral along $C_{Q}$ of a smooth modular form $F$ of weight $2k \in 2\Z$ for $\Gamma$ by
	\begin{align}\label{eq cycle integral}
	\mathcal{C}(F,Q) = D^{\frac{1-k}{2}}\int_{C_{Q}}F(z)Q(z,1)^{k-1}dz.
	\end{align}
	We prove the following identities of cycle integrals.
	
	\begin{Theorem}\label{theorem identities}
		Let $F:\H \to \C$ be a smooth function which transforms like a modular form of weight $2-2k \in 2\Z$ for $\Gamma$. Then the identity
		\begin{align*}
		\mathcal{C}\left(L_{2-2k}F,Q \right) = \mathcal{C}\left(R_{2-2k}F,Q\right) = \overline{\mathcal{C}\left(\xi_{2-2k}F,Q \right)}
		\end{align*}
		of cycle integrals holds, where $L_{2-2k} = -2iy^{2}\frac{\partial}{\partial \overline{z}}$ and $R_{2-2k} = 2i\frac{\partial}{\partial z} + (2-2k)y^{-1}$ are the Maass lowering and raising operators.
		
		Moreover, if $F$ is a weak Maass form of weight $2-2k$ with eigenvalue $\lambda$, then we have
		\begin{align*}
		\mathcal{C}\left(R_{2-2k}^{k-\ell}F,Q\right) &= \big((k+\ell)(k-\ell-1)-\lambda\big)\mathcal{C}\left(R_{2-2k}^{k-\ell-2}F,Q\right), \qquad \ell \leq k-2, \\
		\mathcal{C}\left(L_{2-2k}^{-k-\ell+2}F,Q\right) &= \big((k+\ell)(k-\ell-1)-\lambda\big)\mathcal{C}\left(L_{2-2k}^{-k-\ell}F,Q\right), \qquad\,\,\, \ell \leq -k,
		\end{align*}
		where $L_{2-2k}^{n}$ and $R_{2-2k}^{n}$ are iterated versions of the lowering and raising operators, see Section~\ref{section weak Maass forms}.
	\end{Theorem}
	
	Using the first relation of Theorem~\ref{theorem identities} and then repeatedly applying the second and third one, we obtain the following formulas.
	
	\begin{Corollary}
			Let $F$ be a harmonic weak Maass form of weight $2-2k$ for $\Gamma$. For $k \geq 1$ and all integers $1 \leq j \leq k$ we have the identity
			\begin{align*}
			\mathcal{C}\left(R_{2-2k}^{2j-1}F,Q\right) = 
			\frac{(j-1)!(k-j)!(2k-2)!}{(k-1)!(2k-2j)!}\overline{\mathcal{C}\left(\xi_{2-2k}F, Q\right)}.
			\end{align*}
			Furthermore, for $k \leq 0$ and all integers $0 \leq j \leq |k|$ we have the identity
			\begin{align*}
			\mathcal{C}\left(L_{2-2k}^{2j+1}F,Q\right) = \frac{(2j)!|k|!}{j!(|k|-j)!}\overline{\mathcal{C}\left(\xi_{2-2k}F, Q\right)}. 
			\end{align*}
		\end{Corollary}
	
	If we apply the first identity in the corollary with $j = k$ and use Bol's identity \eqref{eq Bol}, we recover the identity \eqref{eq bgk formula} of Bringmann, Guerzhoy and Kane.
	
	We start with a section on preliminaries about cycle integrals and weak Maass forms. In Section~\ref{section proof}, we prove Theorem~\ref{theorem identities}.

	\section{Cycle Integrals and Weak Maass Forms}\label{section preliminaries}
	
	\subsection{Cycle integrals}\label{section cycle integrals} 
	Let $Q(X,Y) = aX^{2}+bXY+cY^{2}$ be an integral binary quadratic form of non-square discriminant $D = b^{2}-4ac > 0$, and let $\Gamma_{Q}$ be the stabilizer of $Q$ in $\Gamma$. Associated to $Q$ is the semi-circle $S_{Q}$ given by all $z = x+iy \in \H$ satisfying $a|z|^{2}+bx+c = 0$, which we orient counterclockwise if $a > 0$. Let $C_{Q} =\Gamma_{Q}\backslash S_{Q}$ be the associated geodesic in the modular curve $\Gamma \backslash \H$.
	
	We now give an explicit parametrization of the cycle integral in \eqref{eq cycle integral}. Since the cycle integral only depends on the class of $Q$ mod $\Gamma$ and every class contains a quadratic form with positive $a$-entry, we can assume that $a > 0$. Then the two real endpoints $w < w'$ of $S_{Q}$ are given by $w = \frac{-b-\sqrt{D}}{2a}$ and $w' = \frac{-b+\sqrt{D}}{2a}$. The matrix
	\begin{align*}
	\sigma = \frac{a^{\frac{1}{2}}}{D^{\frac{1}{4}}}\begin{pmatrix}w' & w \\ 1 & 1 \end{pmatrix} \in \SL_{2}(\R)
	\end{align*}
	maps $0$ to $w$ and $\infty$ to $w'$, and hence maps the positive imaginary axis (oriented from $i\infty$ to $0$) to $S_{Q}$ (oriented counterclockwise). Furthermore, we have $Q \circ \sigma = [0,-\sqrt{D},0]$. The stabilizer of $[0,-\sqrt{D},0]$ in $\sigma^{-1} \Gamma \sigma$ is generated by the two matrices $\pm\left(\begin{smallmatrix}\varepsilon & 0 \\ 0 & \varepsilon^{-1}\end{smallmatrix}\right)$ for a suitable $\varepsilon > 1$.
	
	\begin{Lemma}\label{lemma parametrization}
		Let $F: \H \to \C$ be a smooth function which transforms like a modular form of weight $2k \in 2\Z$ for $\Gamma$. Then the cycle integral of $F$ along $C_{Q}$ is given by
		\begin{align*}
		\mathcal{C}(F,Q) = (-i)^{k}\int_{1}^{\varepsilon^{2}}F_{\sigma}(iy)y^{k-1}dy,
		\end{align*}
		where $F_{\sigma}= F|_{2k}\sigma$ with the usual weight $2k$ slash operator $|_{2k}$ of $\SL_{2}(\R)$.
		\end{Lemma}
	
	\begin{proof}
		We can parametrize $C_{Q}$ by $y \mapsto \sigma iy$ with $y$ running from $\varepsilon^{2}$ to $1$. Using that $Q(Miy,1) = (ciy+d)^{-2}(Q\circ M)(iy,1)$ and $\frac{d}{dy}Miy = i(ciy+d)^{-2}$ for every matrix $M = \left(\begin{smallmatrix}a & b \\ c & d\end{smallmatrix} \right) \in \SL_{2}(\R)$, we easily obtain the stated formula.
	\end{proof}
	
	\subsection{Weak Maass Forms} \label{section weak Maass forms}
	A weak Maass form of weight $2\kappa \in 2\Z$ for $\Gamma$ is a smooth function $F:\H \to \C$ which is an eigenform of the weight $2\kappa$ Laplace operator
	\begin{align*}
	\Delta_{2\kappa} = -y^{2}\left(\frac{\partial^{2}}{\partial x^{2}}+\frac{\partial}{\partial y^{2}} \right)+2\kappa i y \left(\frac{\partial}{\partial x}+i\frac{\partial}{\partial y} \right),
	\end{align*}
	which transforms like a modular form of weight $2\kappa$ for $\Gamma$ and which has at most linear exponential growth at the cusp. The Maass lowering and raising operators
	\begin{align*}
	L_{2\kappa} = -2iy^{2}\frac{\partial}{\partial \overline{z}}, \qquad R_{2k} = 2i\frac{\partial}{\partial z}+2\kappa y^{-1},
	\end{align*}
	map weak Maass forms of weight $2\kappa$ with eigenvalue $\lambda$ to weak Maass forms of weight $2\kappa-2$ with eigenvalue $\lambda-2\kappa+2$ and weight $2\kappa+2$ with eigenvalue $\lambda+2\kappa$, respectively. We write $L_{2\kappa}^{n} = L_{2\kappa-2n+2}\circ \ldots\circ L_{2\kappa}$ and $R_{2\kappa}^{n} = R_{2\kappa+2n-2}\circ\ldots\circ R_{2\kappa}$ for $n \in \N_{0}$ (with $L_{2\kappa}^{0} = R_{2\kappa}^{0} = \id$) for their iterated versions, which lower or raise the weight by $2n$. The antilinear differential operator
	\begin{align*}
	\xi_{2\kappa}F = y^{2\kappa - 2}\overline{L_{2\kappa}F} = R_{-2\kappa}y^{2\kappa}\overline{F} = 2iy^{2\kappa}\overline{\frac{\partial}{\partial \overline{z}}F} 
	\end{align*}
	maps weak Maass forms of weight $2\kappa$ with eigenvalue $\lambda$ to weak Maass forms of weight $2-2\kappa$ with eigenvalue $\overline{\lambda}$. Furthermore, it defines a surjective map from the space of harmonic weak Maass forms of weight $2\kappa$ to the space of weakly holomorphic modular forms of weight $2-2\kappa$, see \cite{bf}, Theorem~3.7. The above differential operators are related by
	\begin{align}\label{eq delta relation}
	-\Delta_{2\kappa} = \xi_{2-2\kappa}\xi_{2\kappa} = L_{2\kappa+2}R_{2\kappa} + 2\kappa = R_{2\kappa-2} L_{2\kappa}.
	\end{align}	
	Using the relation \eqref{eq delta relation}, the following lemma is easy to prove by induction.
	
	\begin{Lemma}\label{lemma delta}
		For $\kappa,\ell \in \Z$ we have
		\begin{align*}
		\Delta_{2\ell}R_{2\kappa}^{-\kappa+\ell} &= R_{2\kappa}^{-\kappa+\ell}(\Delta_{2\kappa}-(\kappa-\ell)(\kappa+\ell-1)),\qquad \ell \geq \kappa, \\
		\Delta_{2\ell}L_{2\kappa}^{\kappa-\ell} &= L_{2\kappa}^{\kappa-\ell}(\Delta_{2\kappa}-(\kappa-\ell)(\kappa+\ell-1)), \qquad\,\,\,\, \ell \leq \kappa.
		\end{align*}
	\end{Lemma} 
	
	If $\kappa \leq 0$, then the raising operator and the differential operator $\mathcal{D}= \frac{1}{2\pi i}\frac{\partial}{\partial z}$ are related by Bol's identity
	\begin{align}\label{eq Bol}
	\mathcal{D}^{1-2\kappa} = \frac{1}{(-4\pi)^{1-2\kappa}}R_{2\kappa}^{1-2\kappa}.
	\end{align}
	The operator $\mathcal{D}^{1-2\kappa}$ maps harmonic weak Maass forms of weight $2\kappa$ to weakly holomorphic modular forms of weight $2-2\kappa$, see \cite{bor}, Theorem~1.1.
	
	\section{Proof of Theorem~\ref{theorem identities}}\label{section proof}

	The equality $\mathcal{C}\left(L_{2-2k}F,Q \right)= \overline{\mathcal{C}\left(\xi_{2-2k}F,Q \right)}$ becomes obvious if we use the parametrization of the cycle integrals given in Lemma~\ref{lemma parametrization} and write $\xi_{2-2k} = y^{-2k}\overline{L_{2-2k}}$.
	
		Using Lemma~\ref{lemma parametrization}, we see that $\mathcal{C}\left(L_{2-2k}F,Q \right) = \mathcal{C}\left(R_{2-2k}F,Q\right)$ is equivalent to
		\begin{align*}
		\int_{1}^{\varepsilon^{2}}\left(L_{2-2k}F_{\sigma}\right)(iy)y^{-k-1}dy = -\int_{1}^{\varepsilon^{2}}\left(R_{2-2k}F_{\sigma}\right)(iy)y^{1-k}dy.
		\end{align*}
		Here we also used that the lowering and raising operators commute with the slash action of $\SL_{2}(\R)$. A short calculation shows that
		\begin{align*}
		\left(R_{2-2k}F_{\sigma}\right)(iy)y^{1-k} + \left(L_{2-2k}F_{\sigma}\right)(iy)y^{-k-1} = 2\frac{\partial}{\partial y}\left(F_{\sigma}(iy)y^{1-k}\right).
		\end{align*}
		The integral
		\begin{align*}
		\int_{1}^{\varepsilon^{2}}\frac{\partial}{\partial y}\left(F_{\sigma}(iy)y^{1-k} \right)dy = F_{\sigma}(i\varepsilon^{2})\varepsilon^{2-2k} - F_{\sigma} (i) = \left(F_{\sigma}\big|_{2-2k}\left(\begin{smallmatrix}\varepsilon & 0 \\ 0 &\varepsilon^{-1}\end{smallmatrix} \right)\right)(i) - F_{\sigma} (i)
		\end{align*}
		vanishes since $F_{\sigma}$ transforms like a modular form of weight $2-2k$ for $\sigma^{-1} \Gamma \sigma$. This yields the first identity in Theorem~\ref{theorem identities}.
				
		The second and the third identity in Theorem~\ref{theorem identities} easily follow from the first one applied to $R_{2-2k}^{k-\ell-1}F$ and $L_{2-2k}^{-k-\ell+1}F$. For example, for the third identity we obtain
		\begin{align*}
		\mathcal{C}\left(L_{2-2k}^{-k-\ell+2}F,Q\right) = \mathcal{C}\left(L_{2\ell}L_{2-2k}^{-k-\ell+1}F,Q \right) = \mathcal{C}\left(R_{2\ell}L_{2-2k}^{-k-\ell+1}F ,Q\right) 		\end{align*}
		from the first identity in Theorem~\ref{theorem identities}. Now, using \eqref{eq delta relation}, we further compute
		\begin{align*}
		 \mathcal{C}\left(R_{2\ell}L_{2-2k}^{-k-\ell+1}F ,Q\right)= \mathcal{C}\left(R_{2\ell}L_{2\ell+2}L_{2-2k}^{-k-\ell}F ,Q\right) = \mathcal{C}\left(-\Delta_{2\ell+2}L_{2-2k}^{-k-\ell}F ,Q\right),
		\end{align*}
		and then, using Lemma~\ref{lemma delta},
		\begin{align*}
		\mathcal{C}\left(-\Delta_{2\ell+2}L_{2-2k}^{-k-\ell}F ,Q\right)&= \mathcal{C}\left(-L_{2-2k}^{-k-\ell}(\Delta_{2-2k}-(k+\ell)(k-\ell-1))F,Q \right)\\& = \big((k+\ell)(k-\ell-1)-\lambda\big)\mathcal{C}\left(L_{2-2k}^{-k-\ell}F,Q \right).
		\end{align*}
		The computation for the second identity in the theorem is similar, so we leave it to the reader. This finishes the proof of Theorem~\ref{theorem identities}.

\bibliography{references}{}
\bibliographystyle{plain}

\end{document}